\documentclass[a4paper,12pt]{article}
\usepackage{comment}
\usepackage{cite}
\usepackage{amsmath}
\usepackage{amssymb}
\usepackage{amsfonts}
\usepackage[T1]{fontenc}
\usepackage[utf8]{inputenc}
\usepackage{graphicx}
\usepackage{fancyhdr}
\usepackage{float}
\usepackage{xcolor}
\usepackage{authblk}
\usepackage{mathrsfs}
\usepackage{empheq}
\usepackage[hyphens]{url}
\usepackage{hyperref}
\usepackage[]{amsthm}
\usepackage{tikz}
\usetikzlibrary{arrows.meta, decorations.markings,positioning}
\usepackage{pgfplots}
\pgfplotsset{compat=1.18} 

\usepackage{geometry}
\theoremstyle{plain}
\newtheorem{theorem}{Theorem}[section]

\newtheorem{proposition}[theorem]{Proposition}

\theoremstyle{definition}

\theoremstyle{remark}

\begin{document}

\title{A Variational link between the Olech-Opial inequality,
the Wirtinger inequality, and Emden-Fowler equations}

\author[$\dagger$]{Jean-Christophe {\sc Pain}$^{1,2,}$\footnote{jean-christophe.pain@cea.fr}\\
\small
$^1$CEA, DAM, DIF, F-91297 Arpajon, France\\
$^2$Universit\'e Paris-Saclay, CEA, Laboratoire Mati\`ere en Conditions Extr\^emes,\\ 
F-91680 Bruy\`eres-le-Ch\^atel, France
}

\date{}

\maketitle

\begin{abstract}

We establish a structural connection between the classical
Olech-Opial inequality and the Wirtinger inequality.
Using an integral identity involving the mixed energy term
$uu'$, we derive a nonlinear interpolation inequality
linking these two results.
The optimal constant is characterized by a variational problem
whose extremals satisfy an Emden-Fowler equation.
An explicit expression of the optimal constant is obtained
in terms of the Beta function.
This approach provides a natural bridge between mixed-energy
integral inequalities, classical spectral estimates,
and nonlinear boundary value problems.

\end{abstract}

\bigskip


\medskip


\section{Introduction}

Integral inequalities play a fundamental role in the study of
differential equations and variational problems.
Among the classical results in this direction,
the Wirtinger inequality (see, e.g., Refs. \cite{Hardy1952}) provides a sharp estimate relating
the $L^2$ norm of a function to the $L^2$ norm of its derivative
under Dirichlet boundary conditions.
On the other hand, the Olech-Opial inequality appears naturally
in the qualitative theory of nonlinear differential equations
and provides estimates involving the mixed energy term
$\int uu'$.
Although these inequalities originate from different contexts,
they share a common analytical structure.
The goal of this paper is to clarify this relationship
and to show how it leads naturally to a nonlinear interpolation
inequality connected with Emden-Fowler equations.
More precisely, we show that the identity
\[
(u^2)' = 2uu'
\]
provides a natural bridge between mixed-energy estimates and
quadratic energy inequalities.
This observation allows us to derive a nonlinear inequality
whose optimal constant can be obtained through a variational
problem.
The corresponding Euler-Lagrange equation coincides with
the classical Emden-Fowler equation, which allows us to
compute the optimal constant explicitly in terms of the
Beta function.
The paper is organized as follows.
In Section~2 we establish a structural link between the
Olech-Opial inequality (see Refs. \cite{Opial1960,Olech1963}) and the classical Wirtinger inequality.
Using the identity $(u^2)'=2uu'$, we show that the mixed-energy
estimate provided by the Olech-Opial inequality leads naturally
to a Wirtinger-type bound.
Section~3 illustrates how this relationship can be used in the
analysis of nonlinear boundary value problems of Emden-Fowler
type and provides a priori estimates for the energy of solutions.
In Section~4 we introduce a nonlinear interpolation inequality
connecting the previous results and formulate it as a variational
problem whose extremals satisfy an Emden-Fowler equation.
Finally, Section~5 provides an explicit computation of the optimal
constant using the Beta function, which yields a closed-form
expression depending on the exponent $p$. The case of the alternative mean-zero Wirtinger inequality is discussed in Section~6.

\section{A structural link between the Olech-Opial and Wirtinger inequalities}

In this section we establish a precise relationship between the classical
Olech-Opial inequality and the Wirtinger inequality.
Although these inequalities arise in different contexts
(nonlinear differential equations for Olech-Opial and spectral theory
for Wirtinger), we show that they are closely related through an
integral identity involving the mixed energy term $uu'$.

\subsection{The Wirtinger inequality}

Let $u\in H_0^1(0,L)$, i.e.
\[
u(0)=u(L)=0.
\]
The classical Wirtinger inequality, which can be found, for instance, in
Ref. \cite{Hardy1952}, states
\begin{equation}\label{un}
\int_0^L u(x)^2\,dx
\le
\frac{L^2}{\pi^2}
\int_0^L (u'(x))^2\,dx.
\end{equation}
The constant $\frac{L^2}{\pi^2}$ is optimal and equality holds for
\[
u(x)=\sin\!\left(\frac{\pi x}{L}\right).
\]
This inequality is a direct consequence of the spectral
decomposition of the Dirichlet Laplacian.

\subsection{The Olech-Opial inequality}

Let $u\in H^1(0,L)$ with
\[
u(0)=0.
\]
A classical form of the Olech-Opial inequality (see
Refs. \cite{Opial1960,Olech1963,Beesack1962}) reads
\begin{equation}\label{two}
\int_0^L |u(x)u'(x)|\,dx
\le
\frac{L}{2}
\int_0^L (u'(x))^2\,dx.
\end{equation}
Various proofs and refinements of this inequality can be found in
Refs. \cite{Levinson1964,Mallows1965,Calvert1967}.

\begin{proposition}
The constant $\frac{L}{2}$ in Eq.~(\ref{two}) is optimal.
\end{proposition}

\begin{proof}
Since $u(0)=0$, using the representation
\[
u(x)=\int_0^x u'(t)\,dt,
\]
we obtain
\[
|u(x)|
\le
\int_0^x |u'(t)|dt,
\]
and thus we have
\[
|u(x)u'(x)|
\le
|u'(x)|\int_0^x |u'(t)|dt.
\]
Integrating over $(0,L)$ gives
\[
\int_0^L |u(x)u'(x)|dx
\le
\int_0^L |u'(x)|\left(\int_0^x |u'(t)|dt\right)dx.
\]
Setting
\[
F(x) = \int_0^x |u'(t)|\,dt,
\]
the right-hand side becomes
\[
\int_0^L |u'(x)|\,F(x)\,dx.
\]
Observing that
\[
F'(x) = |u'(x)|,
\]
we get
\[
\int_0^L |u'(x)|\,F(x)\,dx
=
\int_0^L F'(x)F(x)\,dx.
\]
This integral can be evaluated explicitly:
\[
\int_0^L F'(x)F(x)\,dx
=
\frac{1}{2}F(L)^2,
\]
giving
\[
\int_0^L |u(x)u'(x)|\,dx
\le
\frac{1}{2}
\left(\int_0^L |u'(t)|\,dt\right)^2.
\]
We now use
\[
\left(\int_0^L |u'(t)|\,dt\right)^2
\le
L \int_0^L (u'(t))^2\,dt,
\]
which is a direct consequence of the Cauchy-Schwarz inequality,
\[
\|f\|_{L^1}^2 \le L \|f\|_{L^2}^2.
\]
Combining the previous estimates, we obtain
\[
\int_0^L |u(x)u'(x)|\,dx
\le
\frac{L}{2}\int_0^L (u'(x))^2 dx.
\]
The constant is optimal and can be approached by sequences
of functions converging to linear profiles.

\end{proof}

\subsection{An integral identity}

The connection between the two inequalities follows from
the elementary identity
\begin{equation}\label{three}
(u^2)' = 2uu'.
\end{equation}
Integrating from $0$ to $x$ gives
\begin{equation}\label{four}
u(x)^2
=
2\int_0^x u(t)u'(t)\,dt.
\end{equation}
This identity shows that the square of the function
can be reconstructed from the mixed energy term.

\subsection{Derivation of a weak Wirtinger inequality}

Using Eq.~(\ref{four}) and the triangle inequality, we find
\[
|u(x)|^2
\le
2\int_0^x |u(t)u'(t)|dt.
\]
Integrating again over $(0,L)$ gives
\[
\int_0^L u(x)^2 dx
\le
2\int_0^L \int_0^x |u(t)u'(t)|dt\,dx.
\]
Changing the order of integration, we get
\[
\int_0^L u(x)^2 dx
\le
2\int_0^L (L-t)|u(t)u'(t)|dt.
\]
Applying the Olech-Opial inequality (see Eq.~(\ref{two})) yields
\[
\int_0^L u(x)^2 dx
\le
L
\int_0^L |u(t)u'(t)|dt
\le
\frac{L^2}{2}
\int_0^L (u'(t))^2 dt,
\]
and we obtain the estimate
\begin{equation}\label{cinq}
\int_0^L u(x)^2 dx
\le
\frac{L^2}{2}
\int_0^L (u'(x))^2 dx.
\end{equation}

\subsection{Comparison with the optimal constant}

The optimal Wirtinger inequality gives
\[
\frac{L^2}{\pi^2}
\]
instead of $\frac{L^2}{2}$.
Since
\[
\frac{L^2}{2}
>
\frac{L^2}{\pi^2},
\]
the estimate of Eq.~(\ref{cinq}) is weaker but remains of the same structural
type.

\subsection{Interpretation}

The previous derivation shows that the Olech-Opial inequality
controls the mixed energy term
\[
\int uu'
\]
while the Wirtinger inequality controls
\[
\int u^2.
\]
The integral identity of Eq.~(\ref{three}) allows one to pass from the mixed
energy to the quadratic energy.
Thus the Olech-Opial inequality can be viewed as a
precursor inequality that leads naturally to a
Wirtinger-type estimate.
This relationship highlights a structural bridge between
inequalities arising in nonlinear differential equations
and classical spectral inequalities.

\section{Application to Emden-Fowler type equations}

We now illustrate how the connection between the Olech-Opial
and Wirtinger inequalities can be used in the analysis of
nonlinear differential equations of Emden-Fowler type.

\subsection{The Emden-Fowler equation}

Consider the boundary value problem
\begin{equation}\label{six}
-u''(x) = \lambda |u(x)|^{p-1}u(x), \qquad x\in (0,L),
\end{equation}
with Dirichlet boundary conditions
\[
u(0)=u(L)=0 ,
\]
where $\lambda>0$ and $p>1$.
Equations of this form, referred to as Emden-Fowler equations, arise in astrophysics,
thermal self-gravitating models and nonlinear diffusion. They are classical in nonlinear analysis (see Ref. \cite{Rabinowitz1986}).

\subsection{Energy identity}

Multiplying the equation by $u(x)$ and integrating over $(0,L)$,
we obtain
\[
-\int_0^L u''(x)u(x)\,dx
=
\lambda \int_0^L |u(x)|^{p+1}dx.
\]
Integrating by parts and using the boundary conditions yields
\begin{equation}\label{sept}
\int_0^L (u'(x))^2 dx
=
\lambda \int_0^L |u(x)|^{p+1}dx.
\end{equation}
Thus the nonlinear energy is directly controlled
by the gradient energy.

\subsection{Control of the nonlinear term}

We now estimate the nonlinear term in terms of the gradient energy. First, we use the elementary interpolation inequality
\[
\int_0^L |u(x)|^{p+1} dx
\le
\|u\|_{L^2(0,L)}^{\,p-1}
\int_0^L u(x)^2 dx.
\]
Next, applying the Wirtinger inequality (see Eq.~(\ref{un})), we obtain
\[
\|u\|_{L^2(0,L)}^2
\le
\frac{L^2}{\pi^2}
\int_0^L (u'(x))^2 dx.
\]
Combining the previous estimates yields
\[
\int_0^L |u(x)|^{p+1} dx
\le
\left(\frac{L^2}{\pi^2}\right)^{\frac{p+1}{2}}
\left(\int_0^L (u'(x))^2 dx\right)^{\frac{p+1}{2}}.
\]
This estimate can be viewed as a one-dimensional Sobolev-type inequality,
derived here without invoking the general Sobolev embedding theorem. Therefore,
\[
\int_0^L |u(x)|^{p+1} dx
\le
\left(\frac{L^2}{\pi^2}\right)^{\frac{p+1}{2}}
\left(\int_0^L (u'(x))^2 dx\right)^{\frac{p+1}{2}}.
\]
The constant $C$ depends only on $p$ and $L$. Substituting into the energy identity of Eq.~(\ref{sept}) yields
\[
\int_0^L (u'(x))^2 dx
\le
\lambda
\left(
\frac{L^2}{\pi^2}
\int_0^L (u'(x))^2 dx
\right)^{\frac{p+1}{2}}.
\]
This derivation shows that the nonlinear estimate follows
directly from the quadratic energy inequality, without
invoking general Sobolev embeddings.

\subsection{A priori estimate: a lower bound on the energy}

Let us set
\[
E=\int_0^L (u'(x))^2 dx.
\]
Then we have
\[
E
\le
\lambda
\left(
\frac{L^2}{\pi^2}
\right)^{\frac{p+1}{2}}
E^{\frac{p+1}{2}}.
\]
Assuming $E>0$, we obtain
\begin{equation}
E^{\frac{p-1}{2}}
\ge
\frac{1}{\lambda}
\left(\frac{\pi^2}{L^2}\right)^{\frac{p+1}{2}}.
\end{equation}
This provides a lower bound on the energy of any
nontrivial solution.

\subsection{Role of the Olech-Opial inequality and interpretation}

The mixed-energy estimate
\[
u(x)^2
=
2\int_0^x u(t)u'(t)dt
\]
combined with the Olech-Opial inequality yields
\[
\int_0^L u(x)^2 dx
\le
\frac{L^2}{2}
\int_0^L (u'(x))^2 dx.
\]
Thus the nonlinear term can also be controlled
using only the mixed energy $\int uu'$.
This shows that the Olech-Opial inequality provides
an alternative route to derive energy estimates
for nonlinear boundary value problems. The analysis above highlights the following hierarchy:
\[
\text{Olech-Opial} \quad \Rightarrow \quad
\text{mixed energy control}
\]
\[
\Downarrow
\]
\[
\text{Wirtinger-type bounds}
\]
\[
\Downarrow
\]
\[
\text{a priori estimates for nonlinear equations}.
\]
Hence the Olech-Opial inequality can be interpreted
as a structural tool that bridges mixed-energy
estimates and spectral inequalities in the study
of nonlinear differential equations such as
the Emden-Fowler equation.

\section{A variational inequality linking Olech-Opial, Wirtinger and Emden-Fowler equations}

In this section we establish a nonlinear interpolation inequality
connecting the classical Olech-Opial inequality and the Wirtinger
inequality. The optimal constant is obtained through a variational
problem and leads naturally to an energy estimate for Emden-Fowler
equations.

\subsection{An interpolation inequality}

Let $u\in H_0^1(0,L)$.
For $p>1$, define

\[
I_p(u)=\int_0^L |u(x)|^{p+1}dx.
\]

\begin{theorem}
For every $p>1$ there exists a constant $C_p(L)>0$ such that

\begin{equation}
\int_0^L |u(x)|^{p+1}dx
\le
C_p(L)
\left(\int_0^L (u'(x))^2dx\right)^{\frac{p+1}{2}}
\end{equation}
for all $u\in H_0^1(0,L)$.
The optimal constant is
\[
C_p(L)=
\sup_{u\in H_0^1(0,L)\setminus\{0\}}
\frac{\int_0^L |u(x)|^{p+1}dx} 
{\left(\int_0^L (u'(x))^2 dx\right)^{\frac{p+1}{2}}}.
\]
\end{theorem}

\subsection{Application to Emden-Fowler equations}

Consider the boundary value problem
\[
-u''=\lambda |u|^{p-1}u,
\qquad
u(0)=u(L)=0.
\]
Multiplying by $u$ and integrating yields
\[
\int_0^L (u'(x))^2 dx
=
\lambda
\int_0^L |u(x)|^{p+1} dx.
\]
Applying the interpolation inequality with constant $C_p(L)$ gives
\[
\int_0^L (u'(x))^2 dx
\le
\lambda C_p(L)
\left(\int_0^L (u'(x))^2 dx\right)^{\frac{p+1}{2}}.
\]
Hence any nontrivial solution must satisfy
\[
\left(\int_0^L (u'(x))^2 dx\right)^{\frac{p-1}{2}}
\ge
\frac{1}{\lambda C_p(L)}.
\]
This provides an explicit lower bound on the energy of
solutions.

\subsection{Variational characterization and Euler-Lagrange equation}

Define the functional
\[
J(u) = \frac{\displaystyle \int_0^L |u(x)|^{p+1} dx}{\left(\displaystyle \int_0^L (u'(x))^2 dx\right)^{(p+1)/2}}.
\]
A maximizer $u_*$ must satisfy the stationarity condition:
\[
\frac{d}{d\varepsilon} J(u_* + \varepsilon v)\Big|_{\varepsilon=0} = 0 \quad \text{for all } v\in H_0^1(0,L).
\]
Set
\[
F(u) = \int_0^L |u(x)|^{p+1} dx, \quad E(u) = \int_0^L (u'(x))^2 dx, 
\quad J(u) = \frac{F(u)}{E(u)^{(p+1)/2}}.
\]
Calculating the first variation gives
\[
\delta J(u)[v] = \frac{\delta F(u)[v] \, E(u)^{(p+1)/2} - F(u) \frac{p+1}{2} E(u)^{(p-1)/2} \delta E(u)[v]}{E(u)^{p+1}} = 0,
\]
which implies
\begin{equation}\label{var-cond}
\delta F(u)[v] = \frac{p+1}{2} \frac{F(u)}{E(u)} \delta E(u)[v].
\end{equation}
The variations of $F$ and $E$ are
\[
\delta F(u)[v] = (p+1)\int_0^L u^p(x) v(x) \, dx, 
\qquad
\delta E(u)[v] = 2 \int_0^L u'(x) v'(x) \, dx.
\]
Substituting into Eq.~\eqref{var-cond} gives
\[
\int_0^L u^p(x) v(x) \, dx = \mu \int_0^L u'(x) v'(x) \, dx, 
\quad \text{with} \quad \mu = \frac{F(u)}{2 E(u)} (p+1).
\]
Integrating by parts on the right-hand side (and using $v(0)=v(L)=0$) leads to the Euler-Lagrange equation
\begin{equation}\label{EL}
-\mu u'' = u^p, \quad u(0)=u(L)=0.
\end{equation}
This is exactly an Emden-Fowler equation, with $\mu$ as a Lagrange multiplier depending on the extremal.

\subsection{Optimal constant}

Let $u_*$ be a maximizer of $J$. Then
\[
C_p(L)=
\frac{\int_0^L |u_*(x)|^{p+1}dx} 
{\left(\int_0^L (u_*'(x))^2 dx\right)^{\frac{p+1}{2}}}.
\]
Using the energy identity
\[
\int_0^L (u_*'(x))^2 dx
=
\mu
\int_0^L |u_*(x)|^{p+1} dx,
\]
we obtain
\[
C_p(L)
=
\mu^{-\frac{p+1}{2}}.
\]
Thus the optimal constant is determined by the first
Dirichlet solution of the Emden-Fowler equation. The Olech-Opial inequality provides a weaker
but structurally similar bound which leads to the
nonlinear interpolation inequality.

\subsection{Relation with classical inequalities}

For $p=1$, the inequality becomes
\[
\int_0^L u(x)^2 dx
\le
C_1(L)\int_0^L (u'(x))^2 dx,
\]
which is precisely the Wirtinger inequality with
\[
C_1(L)=\frac{L^2}{\pi^2}.
\]

\section{Explicit optimal constant via the Beta function}

Let $u \in H_0^1(0,L)$ and $p>1$.  
The best constant $C_p(L)$ in
\[
\int_0^L |u(x)|^{p+1} dx \le C_p(L) \left( \int_0^L (u'(x))^2 dx \right)^{(p+1)/2}
\]
is attained by a positive extremal $u_*$ solving the Emden-Fowler equation
\[
-u_*'' = \mu u_*^p, \quad u_*(0)=u_*(L)=0,
\]
where $\mu>0$ is a Lagrange multiplier associated with the variational problem. In this section we provide an analytical expression of the constant $C_p(L)$. Multiplying the ordinary differential equation by $u_*'$ and integrate once gives
\[
-u_*'' u_*' = \mu u_*^p u_*' \implies \frac12 (u_*')^2 = \frac{\mu}{p+1} \left(A^{p+1} - u_*^{p+1}\right),
\]
where $A = \max u_*$. This is the classical first integral of the Emden-Fowler equation. Rewriting, we have
\[
dx = \frac{du_*}{u_*'} = \sqrt{\frac{p+1}{2\mu}} \frac{du_*}{\sqrt{A^{p+1}-u_*^{p+1}}}.
\]
By symmetry, the half-interval corresponds to the increasing branch from $0$ to $A$:
\[
\frac{L}{2} = \sqrt{\frac{p+1}{2\mu}} \int_0^A \frac{du_*}{\sqrt{A^{p+1}-u_*^{p+1}}}.
\]
Setting $u_* = A t$ with $t \in [0,1]$, we have $du_* = A dt$ and $A^{p+1}-u_*^{p+1} = A^{p+1} (1-t^{p+1})$. Hence, we get
\[
\int_0^A \frac{du_*}{\sqrt{A^{p+1}-u_*^{p+1}}} = A^{-(p-1)/2} \int_0^1 \frac{dt}{\sqrt{1-t^{p+1}}},
\]
and thus
\[
\frac{L}{2} = \sqrt{\frac{p+1}{2\mu}}\, A^{-(p-1)/2} \int_0^1 \frac{dt}{\sqrt{1-t^{p+1}}}.
\]
Recall that the Beta function is defined as
\[
B(a,b) = \int_0^1 t^{a-1} (1-t)^{b-1} \, dt, \quad a,b>0.
\]
Using the standard Beta function identity
\[
\int_0^1 \frac{dt}{\sqrt{1-t^{p+1}}} = \frac{1}{p+1} B\Big( \frac{1}{p+1}, \frac12 \Big),
\]
we obtain
\[
\frac{L}{2} = \sqrt{\frac{p+1}{2\mu}}\, A^{-(p-1)/2} \frac{1}{p+1} B\Big( \frac{1}{p+1}, \frac12 \Big).
\]
Squaring both sides and isolating $\mu$ gives
\[
\mu = \frac{2}{(p+1) L^2} A^{p-1} \left[ B\Big(\frac{1}{p+1}, \frac12\Big) \right]^2.
\]
By symmetry, we get
\[
\int_0^L u_*(x)^{p+1} dx = 2 \int_0^{L/2} u_*(x)^{p+1} dx = 2 \int_0^A \frac{u_*^{p+1}}{u_*'} du_*.
\]
Substitute $u_*' = \sqrt{\frac{2\mu}{p+1} (A^{p+1}-u_*^{p+1})}$ and $u_* = A t$:
\[
\int_0^L u_*^{p+1}(x) dx = 2 \sqrt{\frac{p+1}{2\mu}}\, A^{(p+1)/2} \int_0^1 \frac{t^{p+1} dt}{\sqrt{1-t^{p+1}}}.
\]
The integral can again be expressed as a Beta function:
\[
\int_0^1 \frac{t^{p+1} dt}{\sqrt{1-t^{p+1}}} = \frac{1}{p+1} B\Big( \frac{p+2}{p+1}, \frac12 \Big).
\]
Finally, using
\[
C_p(L) = \frac{\displaystyle\int_0^L u_*^{p+1}(x) dx}{\displaystyle\left(\int_0^L (u_*'(x))^2 dx\right)^{(p+1)/2}} = \mu^{-(p+1)/2} \frac{\displaystyle\int_0^L u_*^{p+1}(x) dx}{\left(\displaystyle\int_0^L u_*^{p+1}(x) dx\right)^{(p+1)/2}} \, ,
\]
we obtain
\[
C_p(L) = \frac{ L^{p+1} }{ (p+1)^{(p+1)/2} } 
\frac{ \left[ B\Big( \frac{1}{p+1}, \frac12 \Big) \right]^p }{ \left[ B\Big( \frac{p+2}{p+1}, \frac12 \Big) \right]^{(p+1)/2} }.
\]

\paragraph{Consistency check.} When $p=1$, $B(1/2,1/2) = \pi$ and $B(3/2,1/2) = \pi/2$, giving
\[
C_1(L) = \frac{L^2}{\pi^2},
\]
which coincides with the optimal Wirtinger constant.

\section{Alternative mean-zero Wirtinger inequality}

In this section, we discuss a variant of the Wirtinger inequality
for functions with zero mean:
\[
\int_0^L u(x)\,dx = 0.
\]

\subsection{Mean-zero Wirtinger inequality}

Let $u \in H^1(0,L)$ with zero mean. Then the inequality reads
\[
\int_0^L u(x)^2\,dx \le \frac{L^2}{4\pi^2} \int_0^L (u'(x))^2\,dx,
\]
with equality for $u(x) = \sin(2\pi x/L)$ or $\cos(2\pi x/L)$.
Compared to the Dirichlet version, the constant is smaller by a factor of $4$.

\subsection{Connection with Olech-Opial}

Using the integral identity $(u^2)' = 2uu'$ as before, we can
derive a weaker estimate
\[
\int_0^L u(x)^2 dx \le \frac{L^2}{2} \int_0^L (u'(x))^2 dx,
\]
which is analogous to the previous derivation but not optimal
for mean-zero functions.

\subsection{Implications for nonlinear problems}

In the study of Emden-Fowler type equations, replacing
the standard Wirtinger constant by the mean-zero constant
$L^2/(4\pi^2)$ leads to sharper energy estimates:
\[
E^{(p-1)/2} \ge \frac{1}{\lambda} \left(\frac{4\pi^2}{L^2}\right)^{(p+1)/2},
\]
where $E = \int_0^L (u')^2 dx$. Thus, the mean-zero formulation provides an alternative
structural pathway linking mixed-energy estimates
(Olech-Opial) and spectral-type inequalities.

\section{Conclusion}

In this paper we have clarified a structural relationship between
the Olech-Opial inequality and the classical Wirtinger inequality.
The key observation is that the identity $(u^2)' = 2uu'$ allows one to pass from estimates involving the mixed energy
term $\int uu'$ to quadratic energy bounds involving $\int u^2$.
This perspective naturally leads to a nonlinear interpolation
inequality whose optimal constant is characterized by a
variational problem.
The corresponding Euler-Lagrange equation coincides with
the Emden-Fowler equation, which provides a direct link
between integral inequalities and nonlinear boundary value
problems.
Moreover, we obtained an explicit formula for the optimal
constant in terms of the Beta function.
This computation highlights the role played by the fundamental
solution of the Emden-Fowler equation as the extremal function
in the associated variational problem.
These results illustrate how mixed-energy inequalities,
spectral estimates and nonlinear differential equations
can be connected within a unified analytical framework.
Possible extensions include weighted versions of the inequality
or higher-dimensional analogues related to Sobolev-type
inequalities.

Interestingly, the extremal functions associated with this
inequality coincide with the fundamental solutions of the
Emden-Fowler equation, which play the role of nonlinear
analogues of the sine function appearing in the classical
Wirtinger inequality.

\end{document}